\documentclass[12pt,a4paper]{amsart}
\usepackage{color}
\usepackage{amssymb,amsmath,amsthm,amstext,amsfonts}
\usepackage[dvips]{graphicx}
\usepackage{psfrag}
\usepackage{url}
\usepackage{amsfonts}
\usepackage{amsmath}
\usepackage{mathrsfs}
\usepackage{graphicx}
\usepackage{amsmath,amsthm,amssymb,amscd}
\usepackage{color}
\usepackage{enumerate}

\usepackage{epsfig}

\makeatletter \@addtoreset{equation}{section} \makeatother

\renewcommand\thetable{\thesection.\@arabic\c@table}

\theoremstyle{plain}
\newtheorem{maintheorem}{Theorem}

\newtheorem{theorem}{Theorem}[section]

\newtheorem{lemma}{Lemma}[section]

\newtheorem{definition}{Definition}[section]

\newtheorem{Thm}{Theorem}[section]
\newtheorem{Lem}[Thm]{Lemma}

\theoremstyle{remark}
\newtheorem{Def}[Thm] {Definition}
\newtheorem{Rem}[Thm] {Remark}

\long\def\begcom#1\endcom{}

\newcommand{\length}{\operatorname{\length}}

\def\length{\operatorname{length}}

\newcommand{\bl} {\begin{lemma}}
\newcommand{\el} {\end{lemma}}
\newcommand{\bt} {\begin{theorem}}
\newcommand{\et} {\end{theorem}}

\newcommand{\bp}{\begin{proof}}
\newcommand{\ep}{\end{proof}}
\newcommand  {\ee} {\end{equation}}
\newcommand  {\beq} {\begin{eqnarray*}}
\newcommand  {\eeq} {\end{eqnarray*}}

\newcommand  {\bd} {\begin{definition}}
\newcommand  {\ed} {\end{definition}}

\def\ep{\noindent{\hfill $\Box$}}

\begin{document}

\title{Laypunov Irregular Points with Distributional Chaos} 

\author{An Chen and Xueting Tian}

\address{Xueting Tian, School of Mathematical Sciences,  Fudan University\\Shanghai 200433, People's Republic of China}
\email{xuetingtian@fudan.edu.cn}
\urladdr{http://homepage.fudan.edu.cn/xuetingtian}

\address{An Chen, School of Mathematical Sciences,  Fudan University\\Shanghai 200433, People's Republic of China}
\email{15210180001@fudan.edu.cn}

\begin{abstract}
It follows from Oseledec Multiplicative Ergodic Theorem (or Kingman¡¯s Sub-additional Ergodic Theorem) that the Lyapunov-irregular set of
points for which the Oseledec averages of a given continuous cocycle diverge has zero measure with respect to any invariant probability measure.
In strong contrast, for any dynamical system $f : X \rightarrow X$ with exponential specification property and a H$\ddot{\text{o}}$lder continuous matrix cocycle $A : X \rightarrow GL(m,R)$, we show here that if there exist ergodic measures with different Lyapunov spectrum, then the Lyapunov-irregular set of $A$ displays distributional chaos of type 1.
\end{abstract}


\keywords{Laypunov exponent, Exponential specification, Distributional chaos, Scrambled set}
\subjclass[2010] {  37C50;  37B20;  37B05; 37D45; 37C45.   }
\maketitle
\section{Introduction}

In page 264 of his book   \cite{MaBook},  Ricardo Ma\~{n}\'{e} wrote: ``In general, (Lyapunov) regular points are very few from the topological point of view - they form a set of first category". Some authors put different interpretations on his statement. For example, Theorem 3.14 of \cite{ABC} by Abdenur, Bonatti and Crovisier interprets ``in general" in the statement as ``for $C^1$-generic diffeomorphisms"; Theorem 1.4 of \cite{Tresidual} by Tian interprets ``in general" as ``a class of cocycles over dynamics with the exponential specification property". In this paper, we follow Tian's interpretation and study the dynamical complexity of Lyapunov-irregular set. We first introduce the cocycles and define the associated Lyapunov exponents.

Cocycles appear naturally in many important problems in dynamics; for instance, derivative cocycles and  Schr$\ddot{\text{o}}$dinger cocycles\cite{Vianabook}. Let $X$ be a compact metric space, $f:X \rightarrow X$ be a homeomorphism and $A:X\rightarrow GL (m,\mathbb{R})$  be a  continuous matrix function. One main object of interest is the asymptotic behavior of the products of $A$ along the orbits of the transformation $f$, called cocycle induced from $A$: for $n>0$
 $$A (x,n):=A (f^{n-1} (x))\cdots A (f (x))A (x)$$
and $$A (x,-n):=A (f^{-n} (x))^{-1}\cdots A (f^{-2} (x))^{-1} A (f^{-1} (x))^{-1}=A (f^{-n}x,n)^{-1}.$$
An important object in understanding the asymptotic behavior of $A (x,n)$ is the Lyapunov exponents associated with the $A (x,n)$. 
 \begin{Def}\cite{MaBook}\label{def2015-Lyapunov-exp}
  We say $x\in X$ to be  (forward) {\it Lyapunov-regular} for $A$,  if there exist numbers $ \chi_1<\chi_2<\cdots<\chi_r\,(r\leq m),$    and an $A-$invariant  decomposition of $\mathbb{R}^m$
$$\mathbb{R}_x^m=G_{ 1} (x)\oplus G_{ 2} (x)\oplus \cdots G_{ r} (x)$$   such that
for any $i=1,\cdots,l$ and any $0\neq v\in G_{ i} (x)$ one has
$$\lim_{n\rightarrow +\infty} \frac1n \log \|A (x,n)v\|=\chi_i.$$ Otherwise, $x$ is called to be {\it Lyapunov-irregular} for $A$. Let $LI(A,f)$ denote the space of all Lyapunov-irregular points for $A$.
\end{Def}


If $m=1,$ $\chi_1$ can be written as Birkhoff ergodic average $$ \lim_{n\rightarrow +\infty}\frac1n\sum_{j=0}^{n-1}\phi(f^j(x))$$ where $\phi(x)=\log\|A (x)\|$. $\phi(x)$ is a continuous function since $A(x)$ is continuous. So, the case of $m=1$ is in fact to study Birkhoff ergodic average. By Birkhoff ergodic theorem, the irregular set for $\phi(x)$  is always of zero measure for any invariant measure, i.e. simple from the perspective of measure. Pesin and Pitskel \cite{Pesinir} are the first to notice the Birkhoof-irregular set displays dynamical complexity from the topological entropy and dimensional perspective in the case of the full shift on two symbols. Then, Barreira, Schmeling, etc. showed that the irregular points can carry full entropy in more general systems(see \cite{BSir,TV,CKS,TDbeta,Thompson2008,PW}). Ruelle used the terminology `historic behavior' in \cite{Ruelle} to describe irregular point and in contrast to dimensional perspective. Takens asked in \cite{Takens} for which smooth dynamical systems the points with historic behavior has positive Lebesgue measure. Moreover, many researchers studied irregular set from topological or geometric or chaotic viewpoint(see \cite{BO,CT,HLO,LW1,LW2,Olsen}). For Lyapunov-irregular set, it is also simple from the perspective of measure by Oseledec Multiplicative Ergodic Theorem. In \cite{Furman}, Furman proved that some smooth cocycles over irrational rotations (which were previously studied by Herman) have a residual set of Lyapunov-irregular points. Then Tian generalized the result in \cite{Tresidual} and study it from the topological entropy\cite{Tentropy}. However, there is no result from the viewpoint of chaos. 

The notion of chaos was first introduced in mathematic language by Li and Yorke in \cite{LY} in 1975. For a dynamical system $(X,f)$, they defined that $(X,f)$ is Li-Yorke chaotic if there is an uncountable scrambled set $S\subseteq X$, where $S$ is called a scrambled set if for any pair of distinct two points $x,y$ of $S$, $$\liminf_{n\to +\infty}d(f^nx,f^ny)=0,\ \limsup_{n\to +\infty}d(f^nx,f^ny)
>0.$$
(We say a pair $x,y$ is distal if $\liminf_{n\to +\infty}d(f^nx,f^ny)>0$). Since then, several refinements of chaos have been introduced and extensively studied. One of the most important extensions of the concept of chaos in sense of Li and Yorke is distributional chaos \cite{SS1994}. The stronger form of chaos has three variants: DC1(distributional chaos of type 1), DC2 and DC3 (ordered from strongest to weakest).   In this paper, we focus on DC1. Readers can refer to \cite{Dwic,SS,SS2} for the definition of DC2 and DC3.
A pair $x,y\in X$ is DC1-scrambled if the following two conditions hold:
$$\forall t>0,\ \limsup_{n\to \infty}\frac{1}{n}|\{i\in [0,n-1]:\ d(f^i(x),f^i(y))<t\}|=1,$$
$$\exists t_0>0,\ \liminf_{n\to \infty}\frac{1}{n}|\{i\in [0,n-1]:\ d(f^i(x),f^i(y))<t_0\}|=0.$$ In other words, the orbits of $x$ and $y$ are arbitrarily close with upper density one, but for some distance, with lower density zero. 

\begin{Def}
 A set $S$ is called a    DC1-scrambled set if any pair of  distinct points in $S$ is DC1-scrambled. A map $f:X\rightarrow X$ is distributional chaos of type 1 if $X$ has an uncountable DC1-scrambled set $S$.
\end{Def}
The distributional chaos is a very famous concept in describing the dynamical complexity. For example, when $X$ is one dimension, the existence of at least one scrambled pair (in the weakest version DC3) implies positive topological entropy \cite{SS1994} or even much more complex dynamics \cite{BSSS}. Distributional chaos, to a certain extent, reveals the topological complexity of trajectories. So, inspired by Ma\~{n}\'{e}'s statement and the results in \cite{Furman,Tresidual,Tentropy}, we show that for a dynamical system with exponential specification(see definition in section 2), Lyapunov-irregular set displays distributional chaos of type 1.

\begin{maintheorem}\label{LIDC-1}
Let $f : X\rightarrow X$ be a homeomorphism of a compact metric space $X$ with exponential specification. Let $A : X \rightarrow GL(m,R)$ be a H\"{o}lder continuous matrix function. Then either all ergodic measures have same Lyapunov spectrum or the Lyapunov-irregular set $LI(A,f)$ contains an uncountable DC1-scrambled set.
\end{maintheorem}

\section{Preliminaries}

\subsection{Oseledec Multiplicative Ergodic Theorem  \cite [Theorem 3.4.4]{BP}\cite{Os}}
Let $f$ be an invertible ergodic measure-preserving transformation of a Lebesgue probability measure space $ (X,\mu).$ Let $A$ be a measurable cocycle whose generator satisfies $\log \|A^\pm (x)\|\in L^1 (X,\mu).$ Then there exist numbers $$\chi_1<\chi_2<\cdots<\chi_r,$$ an $f-$invariant set $\mathcal{R}^\mu$ with $\mu (\mathcal{R}^\mu)=1,$ and an $A-$invariant   decomposition of $\mathbb{R}^m$ for $x\in \mathcal{R}^\mu,$
$$\mathbb{R}_x^m=E_{\chi_1} (x)\oplus E_{\chi_2} (x)\oplus \cdots E_{\chi_r} (x)$$ with $dim E_{\chi_i} (x)=m_i,$ such that
for any $i=1,\cdots,r$ and any $0\neq v\in E_{\chi_i} (x)$ one has
$$\lim_{n\rightarrow \pm\infty} \frac1n \log \|A (x,n)v\|=\chi_i$$
and $$\lim_{n\rightarrow \pm\infty} \frac1n \log det\, A (x,n) =\sum_{i=1}^rm_i\chi_i.$$

 \begin{Def}\label{Def-LyapunovExponents}
The numbers $\chi_1,\chi_2,\cdots,\chi_r$ are called the {\it Lyapunov exponents} of measure $\mu$ for cocycle $A$ and the dimension $m_i$ of the space $E_{\chi_i} (x)$ is called the {\it multiplicity} of the exponent $\chi_i.$ The collection of pairs $$Sp(\mu,A)=\{(\chi_i,m_i):1\leq i \leq r\}$$ is the {\it Lyapunov spectrum} of measure $\mu.$   $\mathcal{R}^\mu$  is called the {\it Oseledec basin} of  $\mu$ and the decomposition  $\mathbb{R}^m=E_{\chi_1} \oplus E_{\chi_2} \oplus \cdots E_{\chi_r} $  is called the {\it Oseledec splitting} of  $\mu$.  
\end{Def}

Note that for any ergodic measure $\mu,$ all the points in the set $\mathcal{R}^\mu$ are  Lyapunov-regular.  By Oseledec's Multiplicative Ergodic theorem and Ergodic Decomposition Theorem, the set   $$\Delta:=\bigcup_{\mu\in \mathcal{M}^e_{f}(X)} \mathcal{R}^\mu$$ is a Borel set with total measure, that is, $\Delta$ has full measure for any invariant Borel probability measure, where $\mathcal{M}^e_{f}(X)$ denotes the space of all ergodic measures.
  In other words, the Lyapunov-irregular set is 
    always   of zero measure for any invariant probability measure. This does not mean that the set of  Lyapunov-irregular points, where the Lyapunov exponents do  not exist, is empty, even if it is completely negligible from the point of view of measure theory.

%
%

\subsection{Lyapunov Exponents and Lyapunov Metric}
In this section let us recall some Pesin-theoretic techniques, which are mainly from   \cite{Kal} (also see \cite{BP}).

Suppose  $f:X\rightarrow X$ to be an invertible map on a compact metric space $X$ and $A:X\rightarrow GL (m,\mathbb{R})$ to  be a  continuous matrix function. For an ergodic measure $\mu$, let  $\chi_1<\chi_2<\cdots<\chi_r $  be the Lyapunov exponents of $\mu,$   $\mathcal{R}^\mu$  be the   Oseledec basin  of  $\mu$ and the decomposition  $\mathbb{R}^m=E_{\chi_1} \oplus E_{\chi_2} \oplus \cdots E_{\chi_r} $  be the Oseledec splitting  of  $\mu$.
We denote the standard scalar product in $\mathbb{R}^m$ by $\langle \cdot,\cdot\rangle $. For a fixed $\epsilon> 0$ and a   point $x\in \mathcal{R}^\mu$,    the {\it $\epsilon-$Lyapunov scalar product  (or metric)} $\langle \cdot,\cdot\rangle _{x,\epsilon}$ in $\mathbb{R}^m$ is defined as follows.

\begin{Def} \label{Def-scalarproduct}
For $u\in E_{\chi_i} (x),\, v\in E_{\chi_j} (x),\,i\neq j$ we define
 $\langle \cdot,\cdot\rangle _{x,\epsilon}=0.$ For $i=1,\cdots,r$ and $u,v\in E_{\chi_i} (x),$ we define
$$\langle \cdot,\cdot\rangle _{x,\epsilon}=m\sum_{n\in\mathbb{Z}}\langle A (x,n)u,A (x,n)v\rangle exp (-2\chi_in-\epsilon |n|).$$
\end{Def}
Note that the series in Definition \ref{Def-scalarproduct} converges exponentially for any   $x\in\mathcal{R}^\mu$. The constant $m$ in front of the conventional formula is introduced for more convenient comparison with the standard scalar product. Usually, $\epsilon$ will be fixed and we will denote $\langle \cdot,\cdot\rangle _{x,\epsilon}$ simply by $\langle \cdot,\cdot\rangle _{x}$
 and call it the {\it Lyapunov scalar product.} The norm generated by this scalar product is called the {\it Lyapunov norm}  and is denoted by $\|\cdot\|_{x,\epsilon}$ or $\|\cdot\|_{x}.$


It should be emphasized that, for any given $\epsilon>0,$ Lyapunov scalar product and Lyapunov norm are defined only for $x\in\mathcal{R}^\mu$.
  They depend only measurably on the point even if the cocycle is H$\ddot{\text{o}}$lder. Therefore, comparison with the standard norm becomes important. The uniform lower bound follow easily from the definition: $$\|u\|_{x,\epsilon}\geq \|u\|.$$ The upper bound is not uniform, but it changes slowly along the   orbits of each $x\in\mathcal{R}^\mu$:
   there exists a measurable function $K_\epsilon (x)$ defined on the set  $\mathcal{R}^\mu$ such that
\begin{eqnarray}\label{eq-different-norm-estimate}
\|u\|\leq \|u\|_{x,\epsilon}\leq K_\epsilon (x)\|u\|\,\,\,\,\,\,\,\,\,\forall x\in\mathcal{R}^\mu,\,\,\forall u\in  \mathbb{R}^m
 \end{eqnarray}

When $\epsilon$ is fixed   it is usually  omitted and write $K (x)=K_\epsilon (x).$ For any $l>1$ we also define the following subsets of $\mathcal{R}^\mu$
\begin{eqnarray}\label{eq-estimate-measure-Pesinblock}
  \mathcal{R}^\mu_{\epsilon,l}=\{x\in \mathcal{R}^\mu: \,\,K_\epsilon (x)\leq l\}.
 \end{eqnarray}
Note that $$\lim_{l\rightarrow \infty}\mu (\mathcal{R}^\mu_{\epsilon,l} )\rightarrow 1.$$ Without loss of generality, we can assume that the set $\mathcal{R}^\mu_{\epsilon,l}$ is compact and that Lyapunov splitting and Lyapunov scalar product are continuous on $\mathcal{R}^\mu_{\epsilon,l}.$ Indeed, by Luzin's theorem we can always find a subset of $\mathcal{R}^\mu_{\epsilon,l}$ satisfying these properties with arbitrarily small loss of measure (for standard Pesin sets these properties are automatically satisfied).

\subsection{Specification and Exponential Specification}
Now we introduce (exponential) specification property. Let $f$ be a homeomorphism of a compact metric space $X$. Denote $B_n(x,\delta):=\{y|\ d(f^ix,f^iy)< \delta,\ 0\leq i\leq n\}$. We say the orbit segments $x,fx,\cdots,f^nx$ and $y,fy,\cdots,f^ny$ are exponentially $\delta$ close with exponent $\lambda$ if
$$d(f^ix,f^iy)< \delta e^{-\lambda \mathrm{min}\{i,n-i\}},\ 0\leq i\leq n.$$
Denote $B_n(x,\delta,\lambda):=\{y|\ y,fy,\cdots,f^ny$ and $x,fx,\cdots,f^nx$ are exponentially $\delta$ close with exponent $\lambda$\}.

\begin{Def}
$f$ is called to have specification property if the following holds: for any $\delta>0$, there is a positive integer $N=N(\delta)$ such that for any $k\geq 1$, any point $x_1,x_2,\cdots,x_k\in X$ and any integers sequences $\{N_i\}_{i=1}^{k},\{M_i\}_{i=1}^{k}$ with $N_1=0, N_j\geq N_{j-1}+M_{j-1}+N$,
$$\bigcap_{i=1}^{k}f^{-N_i}B_{M_i}(x_i,\delta)$$
is not empty.
\end{Def}

Remark that the specification property introduced by Bowen  \cite{Bowen3} (or see \cite{KatHas,DGS}) requires that for any $p\geq N_k+M_k+N,$ $\bigcap_{i=1}^{k}f^{-N_i}B_{M_i}(x_i,\delta)$ contains a periodic point with period $p$. We call this to be   {\it Bowen's   Specification property.}


\begin{Def}\label{definition-of-exp-spe}
$f$ is called to have exponential specification property with exponent $\lambda> 0$ (only dependent on the system $f$ itself) if the following holds: for any $\delta>0$, there is a positive integer $N=N(\delta)$ such that for any $k\geq 1$, any point $x_1,x_2,\cdots,x_k\in X$ and any integers sequences $\{N_i\}_{i=1}^{k},\{M_i\}_{i=1}^{k}$ with $N_1=0, N_j\geq N_{j-1}+M_{j-1}+N$,
$$\bigcap_{i=1}^{k}f^{-N_i}B_{M_i}(x_i,\delta,\lambda)$$
is not empty.
\end{Def}
Further, if for any $p\geq N_k+M_k+N,$ $\bigcap_{i=1}^{k}f^{-N_i}B_{M_i}(x_i,\delta,\lambda)$ contains a periodic point with period $p$, then we say $f$ has Bowen's exponential specification property.

\begin{Rem}\label{Rem-exp-spe}
The dynamical systems with the exponential specification widely exist. For example, every transitive Anosov diffeomorphism  has exponential specification property. For more details and examples, see \cite[Remark 2.3, Example 2.4]{Tresidual} and \cite{Kal}
\end{Rem}

\section{Maximal Lyapunov Exponent and Estimate of The Norm of H$\ddot{\text{o}}$lder Cocycles}
The maximal (or largest) Lyapunov exponent  (or simply, MLE) of $A:X\rightarrow GL (m,\mathbb{R})$ at one point $x\in X$ is defined as the limit $$\chi_{max} (A,x):=\lim_{n\rightarrow \infty}\frac1n{\log\|A (x,n)\|},
   $$ if it exists. In this case $x$ is called to be (forward) {\it Max-Lyapunov-regular}. 
    Otherwise, $x$ is {\it Max-Lyapunov-irregular}. By  Kingman's Sub-additive Ergodic Theorem, for
    any ergodic  measure $\mu$ and $\mu$ a.e. point $x$,  MLE always exists and is constant, denoted by $\chi_{max}(A,\mu)$. From Oseledec Multiplicative Ergodic Theorem (as stated above), it is easy to see that  $\chi_{max}(A,\mu)=\chi_r$ where $\chi_1<\chi_2<\cdots<\chi_r$ are the Lyapunov exponents of $\mu.$ 
Let   $MLI (A,f)$ denote the set of all   Max-Lyapunov-irregular points.
 Then it is of zero measure for any ergodic measure and  by Ergodic Decomposition theorem so does it for all invariant measures.
Now let us recall a general estimate of the norm of $A$ along any orbit segment close to   one orbit of $x\in  \mathcal{R}^\mu$  \cite{Kal}.

 \begin{Lem}\label{Lem-simple-estimate-Lyapunov}     \cite [Lemma 3.1] {Kal}
Let $A$ be an $\alpha-$H$\ddot{\text{o}}$lder cocycle ($\alpha>0$) over a continuous map $f$ of a compact metric space $X$ and let $\mu$ be an ergodic measure for $f$ with the maximal Lypunov exponent $\chi_{max}(A,\mu)=\chi.$ Then for any positive $\lambda$ and $\epsilon$ satisfying $\lambda>\epsilon/ \alpha$ there exists $c>0$ such that for any $n\in\mathbb{N}$, any   point $x\in \mathcal{R}^\mu$ with both $x$ and $f^nx$ in $\mathcal{R}^\mu_{\epsilon,l}$, and any point $y\in X$ such that the orbit segments $x,fx,\cdots,f^nx$ and $y,fy,\cdots,f^n (y)$ are exponentially $\delta$ close with exponent $\lambda$ for some $\delta>0$, we have
 \begin{eqnarray}\label{eq-estimate-simple-Lyapunov-2}
  \|A (y,n)\| \leq l e^{cl\delta^\alpha}e^{n (\chi+\epsilon)}\leq l^2 e^{2n\epsilon+cl\delta^\alpha}\|A (x,n)\|.
 \end{eqnarray}
 The constant $c$ depends only on the cocycle $A$ and on the number $ (\alpha\lambda-\epsilon).$

\end{Lem}

Another lemma is to estimate the growth of vectors in a certain cone $K\subseteq \mathbb{R}^m$ invariant under $A (x,n)$   \cite{Kal}. Let $\chi_1<\chi_2<\cdots<\chi_r$  be the Lyapunov exponents of $\mu.$ Let $x$ be a point in $\mathcal{R}^\mu_{\epsilon,l} $ and $y\in X$ be a point such that the orbit segments $x,fx,\cdots,f^nx$ and $y,fy,\cdots,f^ny$ are exponentially $\delta$ close with exponent $\lambda.$ We denote
$x_i=f^ix$ and $y_i=f^iy,\,i=0,1,\cdots,n.$ For each $i$ we have orthogonal splitting $\mathbb{R}^m=E_i\oplus F_i$ with respect to the Lyapunov norm,  where $E_i$ is the Lyapunov space at $x_i$ corresponding to the maximal  Lyapunov exponent $\chi=\chi_r$ and $F_i$ is the direct sum of all other Lyapunov spaces at $x_i$ corresponding to the Lyapunov exponents less than $\chi.$ For any vector $u\in \mathbb{R}^m$ we denote by $u=u'+u^\perp$ the corresponding splitting with $u'\in E_i$ and $u^\perp\in F_i;$ the choice of $i$ will be clear from the context. To simplify notation, we write $\|\cdot\|_i$ for the Lyapunov norm at $x_i$. For each $i=0,1,\cdots,n$ we consider cones $$K_i=\{u\in \mathbb{R}^m:\,\|u^\bot\|_i\leq \|u'\|_i\}.$$
Note that  for $u\in K_i,$ \begin{eqnarray}\label{eq-u-u'}
\|u\|_i\geq \|u'\|_i\geq\frac1{\sqrt2}\|u\|_i.
\end{eqnarray}
 If all Lyapunov exponent
of $A$ with respect to $\mu$ are equal to $\chi$(that is, $r=1$), one has $F_i=\{0\}, K_i=\mathbb{R}^m$, in this case let
\begin{eqnarray}\label{eq-epsilon00000}
 \epsilon_0(\mu)= \lambda\alpha.
 \end{eqnarray}
If not all Lyapunov exponent of $A$ with respect to $\mu$ are equal to $\chi$ (that is, $r>1$), let $\sigma<\chi$ be the second largest Lyapunov exponent of $A$ with respect to $\mu$, that is, $\sigma=\chi_{r-1}$.
 In this case   set  \begin{eqnarray}\label{eq-epsilon0}
  \epsilon_0(\mu)=\min\{\lambda\alpha, (\chi-\sigma)/2\}.
    \end{eqnarray}

\begin{Lem}\label{Lem-simple-estimate-Lyapunov-2}
In the notation above, for any $\epsilon\in(0,\epsilon_0(\mu))$ and any set $\mathcal{R}^\mu_{\epsilon,l}$, there exist $\delta>0$ such that if  $x,f^nx\in\mathcal{R}^\mu_{\epsilon,l}$  and   the orbit segments $x,fx,\cdots,f^nx$ and $y,fy,\cdots,f^n (y)$ are exponentially $\delta$ close with exponent $\lambda$, then for every $i=0,1,\cdots,n-1$ we have $A (y_i) (K_i)\subseteq K_{i+1}$ and  $\| (A (y_i)u)'\|_{i+1}\geq e^{\chi-2\epsilon}\|u'\|_i$ for any $u\in K_i.$

\end{Lem}
Lemma \ref{Lem-simple-estimate-Lyapunov-2} is a direct corollary of \cite [Lemma 3.3]{Kal}.

\section{Proof of Theorem \ref{LIDC-1}}

\subsection{DC1 of Maximal-Lyapunov-irregularity}

\begin{Thm}\label{Thm-Maximal-DC1-LyapunovIrregular-Cocycle}

Let $f:X\rightarrow X$ be a  continuous map of a compact metric space $X$ with the exponential  specification.     Let $A:X\rightarrow GL (m,\mathbb{R})$  be a $\alpha-$H$\ddot{o}$lder continuous function for some $\alpha>0$. Suppose that
\begin{eqnarray}\label{different-L-spectrum}
\inf_{\mu\in \mathcal{M}^e_{f}(X)}\chi_{max} (A,\mu)<\sup_{\mu\in \mathcal{M}^e_{f}(X)}\chi_{max} (A,\mu)
\end{eqnarray}
Then the Max-Lyapunov-irregular set $MLI(A,f)$ contains an uncountable DC1-scrambled set.

\end{Thm}


\begin{proof}

By \cite[Proposition 4.2]{CT}, we can take a $\nu\in \mathcal M_f^e(X)$ such that $S_\nu$ is minimal and nondegenerate. By (\ref{different-L-spectrum}), there is a $\omega\in \mathcal M_f^e(X)$ such that $$\chi_{max}(A,\nu)\neq \chi_{max}(A,\omega).$$
Without loss of generality, we can assume that $$a=\chi_{max}(A,\nu)> \chi_{max}(A,\omega)=b.$$
Fix a $\tau>0$ such that $a-2\tau>b+\tau.$ Let $\lambda$ be the positive number in the definition of exponential specification. Let $\epsilon\in (0,\mathrm{min}\{\tau,\epsilon_0(\nu)\})$, where $\epsilon_0(\nu)$ is the number w.r.t. measure $\nu$ defined in (\ref{eq-epsilon00000}) or (\ref{eq-epsilon0}). So $\lambda>\epsilon /\alpha$ holds. By Lemma \ref{Lem-simple-estimate-Lyapunov}, there is a constant $c$ such that the consequence of Lemma \ref{Lem-simple-estimate-Lyapunov} for $\omega$ holds. Let $\delta>0$ be small enough such that Lemma \ref{Lem-simple-estimate-Lyapunov-2} applies to $\nu$ and
\begin{equation}\label{c-delta}
c\delta^\alpha<1.
\end{equation}
Let $\delta_1=\delta/2$ and $\delta_{i+1}=\delta_{i}/2, i\in\mathbb{N}^+$. Let $N_i=N(\delta_i)$ be the constant in Definition \ref{definition-of-exp-spe}. For the measures $\nu$ and $\omega$, take $l$ large enough such that
$$\nu(\mathcal R_{\epsilon,l}^\nu)>0,\ \ \omega(\mathcal R_{\epsilon,l}^\omega)>0.$$
Note that $\nu(S_\nu)=1$. Then, by Poincar\'{e} Recurrence theorem, there exist two points $x\in \mathcal R_{\epsilon,l}^\nu\cap S_{\nu}, z \in \mathcal R_{\epsilon,l}^\omega$ and
two increasing integer sequences $\{H'_i\}, \{L'_i\}\nearrow \infty$ such that $f^{H'_i}x\in \mathcal R_{\epsilon,l}^\nu$ and $f^{L'_i}x\in \mathcal R_{\epsilon,l}^\omega$. For any number sequence $\{S_i\}$, we set $\sum_{i=1}^{0}S_i=0$. Let $\{\xi_i\}$ be a strictly decreasing sequence with $\lim_{i\to\infty}\xi_i =0$. Then we can choose two subsequences $\{H_i\}\subseteq\{H'_i\}, \{L_i\}\subseteq\{L'_i\}$ such that
\begin{equation}\label{L-control-1}
\frac{\sum(k)+N_{k+1}}{\sum(k)+N_{k+1}+L_{k+1}}<\xi_{k+1},\ k\geq 1
\end{equation}
and
\begin{equation}\label{H-control-1}
\frac{\sum(k)+N_{k+1}+\sum(k,i)}{\sum(k)+N_{k+1}+\sum(k,i)+H_{k(k+1)/2+i}}<\xi_{k+1},\ k\geq 1,\ 1\leq i\leq k+1,
\end{equation}
where
$$\sum(k):=\sum_{i=1}^{k}L_i+\sum_{i=1}^{k(k+1)/2}H_i+\sum_{i=1}^{k}(i+1)N_i,$$
$$\sum(k,i):=L_{k+1}+iN_i+\sum_{j=1}^{i-1}H_{k(k+1)/2+j}.$$
Denote 
$$\prod(k):=\sum(k)+N_{k+1};$$ 
$$\prod(k,i):=\sum(k)+N_{k+1}+\sum(k,i).$$ 
Then (\ref{L-control-1}) and (\ref{H-control-1}) can be written as
\begin{equation}\label{L-control}
\frac{\prod(k)}{\prod(k)+L_{k+1}}<\xi_{k+1},\ k\geq 1
\end{equation}
and
\begin{equation}\label{H-control}
\frac{\prod(k,i)}{\prod(k,i)+H_{k(k+1)/2+i}}<\xi_{k+1},\ k\geq 1,\ 1\leq i\leq k+1
\end{equation}

Fix $z_0\in X$ and $p\in\{0,1\}^{\infty}$ with $p_1=0$. Let
$$G_1=B_0(z_0,\delta_1,\lambda)\cap f^{-N_1}B_{L_1}(z,\delta_1,\lambda)\cap f^{-(2N_1+L_1)}B_{L_1}(f^{p_1}x,\delta_1,\lambda).$$
By exponential specification, $G_1\neq\emptyset$. Fix a $g_1\in G_1$.
Suppose $g_k$ has been fixed. Let
$$G_{k+1}:=B_{\sum(k)}(g_k,\delta_{k+1},\lambda)\cap f^{-\prod(k)}B_{L_{k+1}}(z,\delta_{k+1},\lambda)\cap(\bigcap_{i=1}^{k+1}f^{-\prod(k,i)}B_{H_{k(k+1)/2+i}}(f^{p_i}x,\delta_{k+1},\lambda)).$$
By exponential specification, $G_{k+1}\neq\emptyset$. Fix a $g_{k+1}\in G_{k+1}$. Hence we construct $\{g_k\}_{k=1}^{\infty}$ by induction. Obviously, $\{g_k\}_{k=1}^{\infty}$ is a Cauchy sequence since $g_{k+1}\in B_{\sum(k)}(g_k,\delta_{k+1},\lambda)$. Without loss of generality, we denote $g(p):=\lim_{k\to\infty}g_k$. By the continuity of $f$, it is easy to check that for any $k\in\mathbb{N^+}$,
\begin{equation}\label{general-LR}
f^{\prod(k)}g(p)\in B_{L_{k+1}}(z,2\delta_{k+1},\lambda)
\end{equation}
and
\begin{equation}\label{minimal-LR}
f^{\prod(k,i)}g(p)\in B_{H_{k(k+1)/2+i}}(f^{p_i}x,2\delta_{k+1},\lambda),\ 1\leq i \leq k+1
\end{equation}
\noindent since $\sum_{i=k}^{\infty}\delta_i\leq 2\delta_k$.
We complete the proof by proving the following three facts:
\begin{description}
\item[(1)] $g(q)\neq g(p)$ for any $q\in\{0,1\}^{\infty}$ with $q_1=0, q\neq p$. Furthermore, $g(p),g(q)$ is a DC1-scrambled pair;
\item[(2)] $\{g(p)\}_{p\in\{0,1\}^{\infty},p_1=0}$ is an uncountable set;
\item[(3)] $g(p)\in MLI(A,f)$.
\end{description}
{\bf(1)}: By $q\neq p$, there is a integer $s\geq 2$ such that $p_s\neq q_s$. Note that $x\in S_\nu$. Then by \cite[Lemma 4.1]{CT}, the pair $x,fx$ is distal. Denote $$\zeta=\inf\{d(f^ix,f^{i+1}x)|\ i\in\mathbb{N}\}.$$
Fixed a $\kappa<\zeta$, we can get an $I_\kappa>s$ such that for any $k\geq I_\kappa$, $4\delta_{k+1}<\zeta-\kappa$. By (\ref{minimal-LR}), for any $k\geq I_\kappa$
$$f^{\prod(k,s)}g(p)\in B_{H_{k(k+1)/2+s}}(f^{p_s}x,2\delta_{k+1},\lambda)$$
and
$$f^{\prod(k,s)}g(q)\in B_{H_{k(k+1)/2+s}}(f^{q_s}x,2\delta_{k+1},\lambda).$$
Then, for any $i\in [0,H_{k(k+1)/2+s}]$
\begin{align*}
&\ \ d(f^{\prod(k,s)+i}g(p),f^{\prod(k,s)+i}g(q))\\
&>d(f^{p_s+i}x,f^{q_s+i}x)-4\delta_{k+1}\\
&>\zeta-4\delta_{k+1}\\
&>\kappa.
\end{align*}
Then
\begin{align*}
& \liminf_{n\to \infty}\frac{1}{n}|\{j\in [0,n-1]:\ d(f^jg(p),f^jg(q))<\kappa\}|\\
\le & \liminf_{k\geq I_\kappa,\ k\to \infty}\frac{|\{j\in [0,\prod(k,s)+H_{k(k+1)/2+s}-1]:\ d(f^jg(p),f^jg(q))<\kappa\}|}{\prod(k,s)+H_{k(k+1)/2+s}}\\
\le & \liminf_{k\geq I_\kappa,\ k\to \infty}\frac{\prod(k,s)}{\prod(k,s)+H_{k(k+1)/2+s}}\\
\overset{(\ref{H-control})}{\le} & \liminf_{k\geq I_\kappa,\ k\to \infty}\xi_{k+1}=0,
\end{align*}
which implies $g(q)\neq g(p)$. On the other hand, For any $t>0$, we can choose $I_t\in\mathbb{N}$ large enough such that $4\delta_{k+1}<t$ holds for any $k\geq I_t$. Note that
$$f^{\prod(k,1)}g(p)\in B_{H_{k(k+1)/2+1}}(f^{p_1}x,2\delta_{k+1},\lambda)$$
and
$$f^{\prod(k,1)}g(q)\in B_{H_{k(k+1)/2+1}}(f^{q_1}x,2\delta_{k+1},\lambda).$$
Note that $p_1=q_1=0$. Then, for any $k\geq I_\kappa$ and any $i\in [0,H_{k(k+1)/2+1}]$,
\begin{align*}
&\ \ d(f^{\prod(k,1)+i}g(p),f^{\prod(k,1)+i}g(q))\\
&<4\delta_{k+1}\\
&<t.
\end{align*}
Then,
\begin{align*}
&\limsup_{n\to \infty}\frac{1}{n}|\{j\in [0,n-1]:\ d(f^ig(p),f^ig(q))<t)\}|\\
\ge & \limsup_{n\to \infty}\frac{1}{n}|\{j\in [0,n-1]:\ d(f^jg(p),f^jg(q))<4\delta_{I_t+1})\}|\\
\ge & \limsup_{k\geq I_t,\ k\to \infty}\frac{|\{j\in [0,\prod(k,1)+H_{k(k+1)/2+1}-1]:\ d(f^jg(p),f^jg(q))<4\delta_{k+1}\}|}{\prod(k,1)+H_{k(k+1)/2+1}}\\
\ge & \limsup_{k\geq I_t,\ k\to \infty}(1-\frac{\prod(k,1)}{\prod(k,1)+H_{k(k+1)/2+1}})\\
\overset{(\ref{H-control})}{\ge} & \limsup_{k\geq I_t,\ k\to \infty}(1-\xi_{k+1})\\
= & 1.
\end{align*}
Thus item {\bf(1)} holds.

{\bf(2)}: Implied by item {\bf(1)} and the fact that $\{0,1\}^{\infty}$ is uncountable.

{\bf(3)}: Let $C=\mathrm{max}_{x\in X}\{||A(x)||,||A^{-1}(x)||\}$. By (\ref{general-LR}), $f^{\prod(k)}g(p)\in B_{L_{k+1}}(z,2\delta_{k+1},\lambda)$. By Lemma \ref{Lem-simple-estimate-Lyapunov} and (\ref{c-delta}),

\begin{align*}
& \frac{1}{\prod(k)+L_{k+1}}\log\|A(g(p),\prod(k)+L_{k+1})\|\\
\leq & \frac{1}{\prod(k)+L_{k+1}}\log C^{\prod(k)}\|A(f^{\prod(k)}g(p),L_{k+1})\|\\
\leq & \frac{1}{\prod(k)+L_{k+1}}\log C^{\prod(k)} le^le^{L_{k+1}(b+\epsilon)}\\
= & \frac{\prod(k)}{\prod(k)+L_{k+1}}\log C+ \frac{L_{k+1}}{\prod(k)+L_{k+1}}(b+\epsilon)+\frac{l+\log l}{\prod(k)+L_{k+1}}\\
\overset{(\ref{L-control})}{\leq} & \xi_{k+1}\log C + (b+\epsilon)+ \frac{l+\log l}{\prod(k)+L_{k+1}}.
\end{align*}

Thus,
\begin{equation}\label{lower-boundary}
\limsup_{k\to\infty}\frac{1}{\prod(k)+L_{k+1}}\log\|A(g(p),\prod(k)+L_{k+1})\|\leq b+\epsilon\leq b+\tau.
\end{equation}
On the other hand, by (\ref{minimal-LR}), $f^{\prod(k,1)}g(p)\in B_{H_{k(k+1)/2+1}}(f^{p_1}x,2\delta_{k+1},\lambda)$. By Lemma \ref{Lem-simple-estimate-Lyapunov-2}, for any $u\in K_0$ with $\|u\|=1$,

\begin{equation}\label{upper}
\|(A(f^{\prod(k,1)}g(p),H_{k(k+1)/2+1})u)'\|_{H_{k(k+1)/2+1}}\geq e^{H_{k(k+1)/2+1}(a-2\epsilon)}\|u'\|_0.
\end{equation}

Together with (\ref{eq-different-norm-estimate}), (\ref{eq-estimate-measure-Pesinblock}), (\ref{eq-u-u'}) and (\ref{upper}), we have

\begin{align*}
& \|A(f^{\prod(k,1)}g(p),H_{k(k+1)/2+1})\|\\
\geq & \|A(f^{\prod(k,1)}g(p),H_{k(k+1)/2+1})u\|\\
\geq & \frac{1}{l}\|A(f^{\prod(k,1)}g(p),H_{k(k+1)/2+1})u\|_{H_{k(k+1)/2+1}}\\
\geq & \frac{1}{l}\|(A(f^{\prod(k,1)}g(p),H_{k(k+1)/2+1})u)'\|_{H_{k(k+1)/2+1}}\\
\geq & \frac{1}{l}e^{H_{k(k+1)/2+1}(a-2\epsilon)}\|u'\|_0\\
\geq & \frac{1}{\sqrt{2}l}e^{H_{k(k+1)/2+1}(a-2\epsilon)}\|u\|_0\\
\geq & \frac{1}{\sqrt{2}l}e^{H_{k(k+1)/2+1}(a-2\epsilon)}\|u\|\\
= & \frac{1}{\sqrt{2}l}e^{H_{k(k+1)/2+1}(a-2\epsilon)}.
\end{align*}

Then,
\begin{align*}
& \frac{\log\|A(g(p),\prod(k,1)+H_{k(k+1)/2+1})\|}{\prod(k,1)+H_{k(k+1)/2+1}}\\
&\\
\geq & \frac{\log C^{-\prod(k,1)}\|A(f^{\prod(k,1)}g(p),H_{k(k+1)/2+1})\|}{\prod(k,1)+H_{k(k+1)/2+1}}\\
&\\
\geq & \frac{-\prod(k,1)\log C+\log\frac{1}{\sqrt{2}l}+H_{k(k+1)/2+1}(a-2\epsilon)}{\prod(k,1)+H_{k(k+1)/2+1}}\\
&\\
\overset{(\ref{H-control})}{\geq} & -\xi_{k+1}\log C+ (1-\xi_{k+1})(a-2\epsilon)+ \frac{\log\frac{1}{\sqrt{2}l}}{\prod(k,1)+H_{k(k+1)/2+1}}.
\end{align*}

Thus,
\begin{equation}\label{upper-boundary}
\liminf_{k\to\infty}\frac{\log\|A(g(p),\prod(k,1)+H_{k(k+1)/2+1})\|}{\prod(k,1)+H_{k(k+1)/2+1}}\geq a-2\epsilon\geq a-2\tau.
\end{equation}

(\ref{lower-boundary}) and (\ref{upper-boundary}) imply that $\frac{1}{n}\log\|A(g(p),n)\|$ diverges as $n\to\infty$, which means $g(p)\in MLI(A,f)$.

\end{proof}

For a cocycle  $A$ and an ergodic measure $\mu$, let $\lambda_1 \geq \lambda_2 \geq\cdots \geq \lambda_m $
 (counted with their multiplicities) denote the Lyapunov exponents of $\mu$ for $A$.
Let $$\Lambda^A_i(\mu)=\sum_{j=1}^i\lambda_j.$$ Then it is easy to see that: for any two ergodic measures $\mu,\nu\in \mathcal{M}^e_{f}(X)$,
\begin{eqnarray}\label{Equivalent-Lyapunov-exponents}
Sp(\mu,A)=Sp(\nu,A)\Leftrightarrow \Lambda^A_i(\mu)=\Lambda^A_i(\nu),\,\forall \,i.
\end{eqnarray}
Let us consider cocycle $\wedge^i A(x,n)$ induced by cocycle $A(x,n)$ on the $i$-fold exterior powers $\wedge^i \mathbb{R}^m$.
 For an  ergodic measure $\mu$, it is standard to see that
 for any $1\leq i\leq m$
\begin{eqnarray}\label{Equivalent-Lyapunov-exponents-2015}\chi_{max}(\wedge^i A,\mu)
=\sum_{j=1}^i\lambda_j=\Lambda^A_i(\mu).
\end{eqnarray}
\medskip

$\mathbf{Proof\ of\ Theorem\ \ref{LIDC-1}}$
Assume that there are two ergodic measures with different Lyapunov spectrum. By  (\ref{Equivalent-Lyapunov-exponents}), there is some $1\leq i \leq m$ such that
\begin{eqnarray*}\label{eq-twodifferent-measure}
  \inf_{\mu\in \mathcal{M}^e_{f}(X)} \Lambda^A_i(\mu)
< \sup_{\mu\in \mathcal{M}^e_{f}(X)} \Lambda^A_i(\mu) .
\end{eqnarray*}
By   (\ref{Equivalent-Lyapunov-exponents-2015}),   one has
\begin{eqnarray*}\label{eq-different-measure}
 \inf_{\mu\in \mathcal{M}^e_{f}(X)}\chi_{max}(\wedge^i A,\mu)<\sup_{\mu\in \mathcal{M}^e_{f}(X)} \chi_{max}(\wedge^i A,\mu).
\end{eqnarray*}
 Then we can apply Theorem \ref{Thm-Maximal-DC1-LyapunovIrregular-Cocycle} to the cocycle
$\wedge^i A(x,n)$ and obtain that the Max-Lyapunov-irregular set of $\wedge^i A$,   $ MLI(\wedge^i A,f),$ contains an uncountable DC1-scrambled set. Note that $LI(A,f)\supseteq   MLI(\wedge^i A,f)$, since  a point Lyapunov-regular for $A$ should be also Max-Lypunov-regular for $\wedge^i A$.  So $LI(A,f)$ contains an uncountable DC1-scrambled set. Now we complete the proof.
\qed

{\bf Acknowledgements.  }   Tian is the corresponding author and is supported by National Natural Science Foundation of China (grant no.   11671093).

\end{document}